\documentclass[12pt]{article}

\usepackage{ucs}
\usepackage{amssymb}
\usepackage{amsthm}
\usepackage{amsmath}
\usepackage{latexsym}
\usepackage[english]{babel}
\usepackage[T1]{fontenc}
\usepackage{ae,aecompl}
\usepackage{mathtools}
\usepackage{amsbsy}
\usepackage{enumitem}
\usepackage{mathrsfs}
\usepackage{multicol}
\usepackage{bbm}
\usepackage{hyphenat}
\usepackage[left=2cm,right=2cm,top=2cm,bottom=2cm,bindingoffset=0cm]{geometry}
\usepackage[hidelinks]{hyperref}

\title{DP-colorings of graphs with high chromatic number}
\date{}
\author{
	Anton~Bernshteyn\thanks{Department of Mathematics, University of Illinois at Urbana--Champaign, IL, USA, \href{mailto:bernsht2@illinois.edu}{\texttt{bernsht2@illinois.edu}}. Research of this author is supported by the Illinois Distinguished Fellowship.}
	\and Alexandr Kostochka\thanks{Department of Mathematics, University of Illinois at Urbana--Champaign, IL, USA and
		Sobolev Institute of Mathematics, Novosibirsk 630090, Russia, \href{mailto:kostochk@math.uiuc.edu}{\texttt{kostochk@math.uiuc.edu}}. Research of this author is supported in part by NSF grant
 DMS-1600592 and grants 15-01-05867 and 16-01-00499 of the Russian Foundation for Basic Research.}
\and Xuding Zhu\thanks{Department of Mathematics, Zhejiang Normal University,
	Jinhua, China, \href{mailto:xudingzhu@gmail.com}{\texttt{xudingzhu@gmail.com}}. Research of this author is supported in part by CNSF grant 11571319.}}

\newtheorem{theo}{Theorem}[section]

\newtheorem{corl}[theo]{Corollary}

\newtheorem{claim}{Claim}[theo]
\newtheorem{smallcorl}[claim]{Corollary}

\newcommand*{\myproofname}{Proof}
\newenvironment{claimproof}[1][\myproofname]{\begin{proof}[#1]}{\end{proof}}

\theoremstyle{definition}
\newtheorem{defn}[theo]{Definition}

\newtheorem{remk}[theo]{Remark}

\newcommand{\0}{\varnothing}
\newcommand{\set}[1]{\{#1\}}

\newcommand{\N}{\mathbb{N}}
\newcommand{\Z}{\mathbb{Z}}

\renewcommand{\epsilon}{\varepsilon}
\renewcommand{\phi}{\varphi}
\renewcommand{\theta}{\vartheta}

\newcommand{\powerset}[1]{\operatorname{Pow}(#1)}

\newcommand{\J}{\mathbf{J}}

\numberwithin{equation}{section}

\makeatletter
\newcommand{\neutralize}[1]{\expandafter\let\csname c@#1\endcsname\count@}
\makeatother

\begin{document}
	
	\maketitle
	
	\begin{abstract}
		 \emph{DP-coloring} is a generalization of list coloring introduced recently by Dvo\v r\' ak and Postle~\cite{DP}. 
We prove that for every $n$-vertex graph $G$  whose  chromatic number $\chi(G)$ is ``close'' to~$n$, the DP-chromatic number of $G$ equals $\chi(G)$.
 ``Close'' here means $\chi(G)\geq n-O(\sqrt{n})$, and we also show that this lower bound is best possible (up to the constant factor in front of~$\sqrt{n}$), in contrast to the case of list coloring.
	\end{abstract}
	

	\section{Introduction}
	
We use standard notation. In particular, $\N$  denotes the set of all nonnegative integers. For a set~$S$, $\powerset{S}$ denotes the power set of $S$, i.e., the set of all subsets of $S$.
All graphs considered here are finite, undirected, and simple. 
For a graph~$G$, $V(G)$ and $E(G)$ denote the vertex and the edge sets of $G$, respectively.
 For a set $U \subseteq V(G)$, $G[U]$ is the subgraph of $G$ induced by $U$. Let $G - U \coloneqq G[V(G) \setminus U]$, and for $u \in V(G)$, 
let $G - u \coloneqq G - \set{u}$. For  $U_1$, $U_2 \subseteq V(G)$, let $E_G(U_1, U_2) \subseteq 
E(G)$ denote the set of all edges in $G$ with one endpoint in $U_1$ and the other one in $U_2$. 
For $u \in V(G)$, $N_G(u)\subset V(G)$ denotes the set of all neighbors of $u$, and $\deg_G(u) \coloneqq |N_G(u)|$ is the \emph{degree} of~$u$ in $G$. We use $\delta(G)$ to denote the \emph{minimum degree} of $G$, i.e., $\delta(G) \coloneqq \min_{u \in V(G)} \deg_G(u)$. For  $U \subseteq V(G)$, let $N_G(U) \coloneqq \bigcup_{u \in U} N_G(u)$. To simplify notation, we write $N_G(u_1, \ldots, u_k)$ instead of $N_G(\set{u_1, \ldots, u_k})$.
A set $I \subseteq V(G)$ is {\em independent} if $I \cap N_G(I) = \0$, i.e., if $uv \not \in E(G)$ for all $u$, $v \in I$. 
We denote the family of all independent sets in a graph $G$ by $\mathcal{I}(G)$. The complete $k$-vertex graph  is denoted by~$K_k$.
		
	\subsection{The Noel--Reed--Wu Theorem for list coloring}
	
	Recall that a \emph{proper coloring} of a graph $G$ is a function $f \colon V(G) \to Y$, where $Y$ is a set of \emph{colors}, 
such that  $f(u) \neq f(v)$ for every edge $uv \in E(G)$. The smallest $k \in \N$ such that there exists a proper coloring $f \colon V(G) \to Y$ with $|Y| = k$ is called the \emph{chromatic number} of $G$ and is denoted by $\chi(G)$.
	
\emph{List coloring} was introduced independently by Vizing~\cite{Vizing} and Erd\H os, Rubin, and Taylor~\cite{ERT}.
 A \emph{list assignment} for a graph $G$ is a function $L \colon V(G) \to \powerset{Y}$, where $Y$ is a set. For each $u \in V(G)$, the set $L(u)$ is called the \emph{list} of $u$, and its elements are the \emph{colors} \emph{available} for~$u$. A~proper coloring $f \colon V(G) \to Y$ 
is called an \emph{$L$-coloring} if  $f(u) \in L(u)$ for each $u \in V(G)$. 
The \emph{list chromatic number} $\chi_\ell(G)$ of $G$  is the smallest $k \in \N$ such that~$G$ is $L$-colorable for each  list assignment $L$
 with $|L(u)| \geq k$ for all $u \in V(G)$. 
 It is an immediate consequence of the definition that $\chi_\ell(G) \geq \chi(G)$ for every graph $G$.
 
 It is well-known (see, e.g.,~\cite{ERT, Vizing}) that the list chromatic number of a graph can significantly exceed its ordinary chromatic number. Moreover, there exist $2$-colorable graphs with arbitrarily large list chromatic numbers. On the other hand, Noel, Reed, and Wu~\cite{NRW}
established the following result, which was conjectured by Ohba~\cite[Conjecture~1.3]{Ohba}:

	\begin{theo}[Noel--Reed--Wu~\cite{NRW}]\label{theo:NRW}
Let $G$ be an $n$-vertex graph with $\chi(G)\geq (n-1)/2$. Then $\chi_\ell(G)=\chi(G)$.
	\end{theo}

	The following construction was first studied by Ohba~\cite{Ohba} and Enomoto, Ohba, Ota, and Sakamoto~\cite{EOOS}. For a graph $G$ and $s \in \N$, let $\J(G,s)$ denote the {\em join} of $G$ and a copy of $K_s$, i.e., the graph obtained from $G$ by adding $s$ new vertices that are adjacent to every vertex in $V(G)$ and to each other. It is clear from the definition that for all $G$ and $s$, $\chi(\J(G,s))=\chi(G)+s$. Moreover, we have $\chi_\ell(\J(G, s)) \leq \chi_\ell(G) + s$; however, this inequality can be strict. Indeed, Theorem~\ref{theo:NRW} implies that for every graph $G$ and every $s \geq |V(G)| - 2 \chi(G) - 1$,
	$$
		\chi_\ell(\J(G, s)) = \chi(\J(G, s)),
	$$
	even if $\chi_\ell(G)$ is much larger than $\chi(G)$. In view of this observation, it is interesting to consider the following parameter:
	\begin{equation}\label{eq:list_Zhu}
		Z_\ell(G) \coloneqq \min \set{s \in \N \,:\, \chi_\ell(\J(G, s)) = \chi(\J(G, s))},
	\end{equation}
	i.e., the smallest $s \in \N$ such that the list and the ordinary chromatic numbers of $\J(G, s)$ coincide. The parameter $Z_\ell(G)$ was explicitly defined by Enomoto, Ohba, Ota, and Sakamoto in~\cite[page~65]{EOOS} (they denoted it $\psi(G)$). Recently, Kim, Park, and Zhu (personal communication, 2016) obtained new lower bounds on $Z_\ell(K_{2,n})$, $Z_\ell(K_{n,n})$, and $Z_\ell(K_{n,n,n})$. One can also consider, for $n \in \N$,
\begin{equation}\label{eq:list_Zhu_n}
	Z_\ell(n) \coloneqq \max \set{Z_\ell(G) \,:\, |V(G)| = n}.
\end{equation}
The parameter $Z_\ell(n)$ is closely related to the Noel--Reed--Wu Theorem, since, by definition, there exists a graph $G$ on $n + Z_\ell(n) - 1$ vertices whose ordinary chromatic number is at least $Z_\ell(n)$ and whose list and ordinary chromatic numbers are distinct. The finiteness of $Z_\ell(n)$ for all $n \in \N$ was first established by Ohba~\cite[Theorem~1.3]{Ohba}. Theorem~\ref{theo:NRW} yields an upper bound $Z_\ell(n) \leq n - 5$ for all $n \geq 5$; on the other hand, a result of Enomoto, Ohba, Ota, and Sakamoto~\cite[Proposition~6]{EOOS} implies that $Z_\ell(n) \geq n - O(\sqrt{n})$.

\subsection{DP-colorings and the results of this paper}

The goal of this note is to study  analogs of $Z_\ell(G)$ and $Z_\ell(n)$ for the generalization of list coloring that was recently introduced by Dvo\v r\' ak and Postle~\cite{DP}, which we call \emph{DP-coloring}. 
 Dvo\v r\' ak and Postle invented DP-coloring to attack an open problem on list coloring of planar graphs with no cycles of certain lengths. 
	
	\begin{defn}
		Let $G$ be a graph. A \emph{cover} of $G$ is a pair $(L, H)$, where $H$ is a graph and $L \colon V(G) \to \powerset{V(H)}$ is a function, with the following properties:
		\begin{itemize}
			\item[--] the sets $L(u)$, $u \in V(G)$, form a partition of $V(H)$;
			\item[--] if $u$, $v \in V(G)$ and $L(v) \cap N_H(L(u)) \neq \0$, then $v \in \set{u} \cup N_G(u)$; 
			\item[--] each of the graphs $H[L(u)]$, $u \in V(G)$, is complete;
			\item[--] if $uv \in E(G)$, then $E_H(L(u), L(v))$ is a matching (not necessarily perfect and possibly empty).
		\end{itemize}
	\end{defn}
	
	\begin{defn}
		Let $G$ be a graph and let $(L, H)$ be a cover of $G$. An \emph{$(L, H)$-coloring} of $G$ is an independent set $I \in \mathcal{I}(H)$ of size $|V(G)|$. Equivalently, $I \in \mathcal{I}(H)$ is an $(L, H)$-coloring of $G$ if $|I \cap L(u)| = 1$ for all $u \in V(G)$.
	\end{defn}
	
	\begin{remk}
		Suppose that $G$ is a graph, $(L, H)$ is a cover of $G$, and $G'$ is a subgraph of $G$. In such situations, we will allow a slight abuse of terminology and speak of $(L, H)$-colorings of $G'$ (even though, strictly speaking, $(L, H)$ is not a cover of $G'$).
	\end{remk}
	
	The \emph{DP-chromatic number} $\chi_{DP}(G)$ of a graph $G$  is the smallest $k \in \N$ such that $G$ is $(L,H)$-colorable for each cover $(L, H)$ with $|L(u)| \geq k$ for all $u \in V(G)$.
	
	To show that DP-colorings indeed generalize list colorings, consider a graph $G$ and a list assignment $L$ for $G$. Define a graph $H$ as follows: Let $V(H) \coloneqq \set{(u, c)\,:\, u \in V(G) \text{ and } c \in L(u)}$ and let
	$$
		(u_1, c_1)(u_2, c_2) \in E(H) \,\vcentcolon\Longleftrightarrow\, (u_1 = u_2 \text{ and } c_1 \neq c_2) \text{ or } (u_1u_2 \in E(G) \text{ and }c_1 = c_2).
	$$
	For $u \in V(G)$, let $\hat{L}(u) \coloneqq \set{(u, c)\,:\, c \in L(u)}$. Then $(\hat{L}, H)$ is a cover of $G$, and there is a one-to-one correspondence between $L$-colorings and $(\hat{L}, H)$-colorings of $G$. Indeed, if $f$ is an $L$-coloring of $G$, then the set $I_f \coloneqq \set{(u, f(u))\,:\, u \in V(G)}$ is an $(\hat{L}, H)$-coloring of~$G$. Conversely, given an $(\hat{L}, H)$-coloring $I$ of $G$, we can define an $L$-coloring $f_I$ of $G$ by the property $(u, f_I(u)) \in I$ for all $u \in V(G)$. Thus, list colorings form a subclass of DP-colorings. In particular, $\chi_{DP}(G) \geq \chi_\ell(G)$ for each graph $G$.
	
	Some upper bounds on list-chromatic numbers hold for DP-chromatic numbers as well. For example,  $\chi_{DP}(G) \leq d+1$ for any $d$-degenerate graph $G$. Dvo\v r\'ak and Postle~\cite{DP} pointed out that Thomassen's bounds~\cite{Thomassen1, Thomassen2}
on the list chromatic numbers of planar graphs hold also for their DP-chromatic numbers; in particular,
 $\chi_{DP}(G) \leq 5$  for every planar graph $G$.
  On the other hand, there are also some striking differences between DP- and list coloring. For instance, even cycles are $2$-list-colorable, while
  their DP-chromatic number is $3$; in particular, the orientation theorems of Alon--Tarsi~\cite{AT} and the Bondy--Boppana--Siegel Lemma (see~\cite{AT}) do not extend to DP-coloring (see~\cite{BK} for further examples of differences between list and DP-coloring).
  
	By analogy with~\eqref{eq:list_Zhu} and \eqref{eq:list_Zhu_n}, we consider the parameters
	$$
		Z_{DP}(G) \coloneqq \min \set{s \in \N \,:\, \chi_{DP}(\J(G, s)) = \chi(\J(G, s))},
	$$
	and
	$$
	Z_{DP}(n) \coloneqq \max \set{Z_{DP}(G) \,:\, |V(G)| = n}.
	$$
	Our main result asserts that for all graphs $G$, $Z_{DP}(G)$ is finite:
  	
	\begin{theo}\label{theo:main}
	Let $G$ be a graph with $n$ vertices, $m$ edges, and chromatic number $k$. Then $Z_{DP}(G) \leq 3m$. Moreover, if $\delta(G) \geq k-1$, then $$Z_{DP}(G) \leq 3m - \frac{3}{2}(k-1)n.$$
	\end{theo}
	
	\begin{corl}\label{corl:Z_n}
		For all $n \in \N$, $Z_{DP}(n) \leq 3n^2/2$.
	\end{corl}
	
	Note that the upper bound on $Z_{DP}(n)$ given by Corollary~\ref{corl:Z_n} is quadratic in $n$, in contrast to the linear upper bound on $Z_\ell(n)$ implied by Theorem~\ref{theo:NRW}. Our second result shows that the order of magnitude of $Z_{DP}(n)$ is indeed quadratic:

	\begin{theo}\label{theo:lower}
			For all $n \in \N$, $Z_{DP}(n) \geq n^2/4 - O(n)$.	
	\end{theo}	
	
	Corollary~\ref{corl:Z_n} and Theorem~\ref{theo:lower} also yield the following analog of Theorem~\ref{theo:NRW} for DP-coloring:
	
	\begin{corl}\label{corl:NRW_DP}
		For $n \in \N$, let $r(n)$ denote the minimum $r \in \N$ such that for every $n$-vertex graph~$G$ with $\chi(G)\geq r$, we have $\chi_{DP}(G)=\chi(G)$.
		Then
		$$
			n - r(n) = \Theta(\sqrt{n}).
		$$
	\end{corl}

We prove Theorem~\ref{theo:main} in Section~\ref{sec:ma} and Theorem~\ref{theo:lower} in Section~\ref{sec:low}. The derivation of Corollary~\ref{corl:NRW_DP} from Corollary~\ref{corl:Z_n} and Theorem~\ref{theo:lower} is straightforward; for completeness, we include it at the end of Section~\ref{sec:low}.
	
	\section{Proof of Theorem~\ref{theo:main}}\label{sec:ma}

	For a graph $G$ and a finite set $A$ disjoint from $V(G)$, let $\J(G, A)$ denote the graph with vertex set $V(G) \cup A$ obtained from $G$ be adding all edges with at least one endpoint in $A$ (i.e., $\J(G, A)$ is a concrete representative of the isomorphism type of $\J(G, |A|)$).

First we prove the following more technical version of Theorem~\ref{theo:main}:
	
		
		\begin{theo}\label{theo:upper_bound}
			Let $G$ be a $k$-colorable graph. Let $A$ be a finite set disjoint from $V(G)$ and let
			 $(L,H)$ be a cover of $\J(G, A)$ such that for all $a \in A$, $|L(a)| \geq |A| + k$. Suppose that \begin{equation}\label{eq:IH}
				|A| \,\geq\, \frac{3}{2}\sum_{v \in V(G)} \max \set{\deg_G(v) + |A| - |L(v)| + 1,\,0}.
			\end{equation}
			Then $\J(G,A)$ is $(L,H)$-colorable.
		\end{theo}
		\begin{proof}
			For a graph $G$, a set $A$ disjoint from $V(G)$, a cover $(L, H)$ of $\J(G, A)$, and a vertex $v \in V(G)$, let
			$$
			\sigma(G, A, L, H, v) \coloneqq \max \set{\deg_G(v) + |A| - |L(v)| + 1,\,0}
			$$
			and
			$$
			\sigma(G, A, L, H) \coloneqq \sum_{v \in V(G)} \sigma(G, A, L, H, v).
			$$
			
			Assume, towards a contradiction, that a tuple $(k, G, A, L, H)$ forms a counterexample which minimizes $k$, then $|V(G)|$, and then $|A|$. For brevity, we will use the following shortcuts:
			$$
			\sigma(v) \coloneqq \sigma(G, A, L, H, v); \;\;\; \sigma \coloneqq \sigma(G, A, L, H).
			$$
			Thus,~\eqref{eq:IH} is equivalent to
			$$
				|A| \geq \frac{3\sigma}{2}.
			$$
			
			Note that $|V(G)|$ and $|A|$ are both positive. Indeed, if $V(G) = \0$, then $\J(G, A)$ is just a clique with vertex set $A$, so its DP-chromatic number is $|A|$. If, on the other hand, $A = \0$, then~\eqref{eq:IH} implies that $|L(v)| \geq \deg_G(v) + 1$ for all $v \in V(G)$, so an $(L, H)$-coloring of $G$ can be constructed greedily. Furthermore, $\chi(G) = k$, since otherwise we could have used the same $(G, A, L, H)$ with a smaller value of $k$.
			
			\begin{claim}\label{claim:critical}
				For every $v \in V(G)$, the graph $\J(G - v, A)$ is $(L,H)$-colorable.
			\end{claim}
			\begin{claimproof}
				Consider any $v_0 \in V(G)$ and let $G' \coloneqq G - v_0$. For all $v \in V(G')$, $\deg_{G'}(v) \leq \deg_G(v)$, and thus
				$\sigma(G', A, L, H, v) \leq \sigma(v)$. Therefore,
				$$
				\frac{3}{2}\sigma (G', A, L, H)\,\leq\,   {3\sigma \over 2}\,\leq \,|A|.
				$$
				By the minimality of $|V(G)|$, the conclusion of Theorem~\ref{theo:upper_bound} holds for $(k, G', A, L, H)$, i.e., $\J(G', A)$ is $(L,H)$-colorable, as claimed.
			\end{claimproof}
			
			\begin{smallcorl}\label{corl:small_lists}
				For every $v \in V(G)$,
				$$\sigma(v) = \deg_G(v) + |A| - |L(v)| + 1 > 0.$$
			\end{smallcorl}
			\begin{claimproof}
				Suppose that for some $v_0 \in V(G)$,
				$$
				\deg_G(v_0) + |A| - |L(v_0)| + 1 \leq 0,
				$$
				i.e.,
				$$
				|L(v_0)| \geq \deg_G(v_0) + |A| + 1.
				$$
				Using Claim~\ref{claim:critical}, fix any $(L,H)$-coloring $I$ of $\J(G - v_0, A)$. Since $v_0$ still has at least
				$$
				|L(v_0)| - (\deg_G(v_0) + |A|) \geq 1
				$$
				available colors, $I$ can be extended to an $(L,H)$-coloring of $\J(G, A)$ greedily; a contradiction.
			\end{claimproof}
			
			\begin{claim}\label{claim:no_neighbor}
				For every $v \in V(G)$ and $x \in \bigcup_{a \in A} L(a)$, there is $y \in L(v)$ such that $xy \in E(H)$.
			\end{claim}
			\begin{claimproof}
				Suppose that for some $a_0 \in A$, $x_0 \in L(a_0)$, and $v_0 \in V(G)$, we have $L(v_0) \cap N_H(x_0) = \0$. Let $A' \coloneqq A \setminus \set{a_0}$, and for every $w \in V(G) \cup A'$, let $L'(w) \coloneqq L(w) \setminus N_H(x_0)$. Note that for all $a \in A'$, $|L'(a)| \geq |A'| + k$, and for all $v \in V(G)$, $\sigma(G, A', L', H, v) \leq \sigma(v)$. Moreover, by the choice of~$x_0$, $|L'(v_0)| = |L(v_0)|$, which, due to Corollary~\ref{corl:small_lists}, yields
				$ \sigma(G, A', L', H, v_0) \leq \sigma(v_0)-1$.
				This implies $\sigma(G, A', L', H) \leq \sigma - 1$, and thus
				$$
				\frac{3}{2} \sigma(G, A', L', H) \leq {3(\sigma - 1) \over 2} \leq |A|-\frac{3}{2} < |A'|.
				$$
				By the minimality of $|A|$, the conclusion of Theorem~\ref{theo:upper_bound} holds for $(k, G, A', L', H)$, i.e., the graph $\J(G, A')$ is $(L', H)$-colorable. By the definition of~$L'$, for any $(L', H)$-coloring $I$ of $\J(G, A')$, $I \cup \set{x_0}$ is an $(L,H)$-coloring of $\J(G, A)$. This is a~contradiction.
			\end{claimproof}
			
			\begin{smallcorl}\label{corl:k_geq_2}
				$k \geq 2$.
			\end{smallcorl}
			\begin{claimproof}
				Let $v \in V(G)$ and consider any $a \in A$. Since, by Claim~\ref{claim:no_neighbor}, each $x \in L(a)$ has a neighbor in $L(v)$, we have
				$$
				|L(v)| \geq |L(a)| \geq |A| + k.
				$$
				Using Corollary~\ref{corl:small_lists}, we obtain
				$$
					0 \leq \deg_G(v) + |A| - |L(v)| \leq \deg_G(v) - k,
				$$
				i.e., $\deg_G(v) \geq k$. Since $V(G) \neq \0$, $k \geq 1$, which implies $\deg_G(v) \geq 1$. But then $\chi(G) \geq 2$, as desired.
			\end{claimproof}
			
			\begin{claim}\label{claim:walk}
				$H$ does not contain a walk of the form $x_0- y_0-x_1-y_1-x_2$, where
				\begin{itemize}
					\item $x_0$, $x_1$, $x_2 \in \bigcup_{a \in A} L(a)$;
					\item $y_0$, $y_1 \in \bigcup_{v \in V(G)} L(v)$;
					\item $x_0 \neq x_1\neq x_2$ and $y_0 \neq y_1$ (but it is possible that $x_0 = x_2$);
					\item the set $\set{x_0, x_1, x_2}$ is independent in $H$.
				\end{itemize}
			\end{claim}
			\begin{claimproof}
				Suppose that such a walk exists and let $a_0$, $a_1$, $a_2 \in A$ and $v_0$, $v_1 \in V(G)$ be such that $x_0 \in L(a_0)$, $y_0 \in L(v_0)$, $x_1 \in L(a_1)$, $y_1 \in L(v_1)$, and $x_2 \in L(a_2)$. Let $A' \coloneqq A \setminus \set{a_0, a_1, a_2}$, and for every $w \in V(G) \cup A'$, let $L'(w) \coloneqq L(w) \setminus N_H(x_0, x_1, x_2)$. Since $\set{x_0, x_1, x_2}$ is an independent set, for all $a \in A'$, $|L'(a)| \geq |A'| + k$, while for all $v \in V(G)$, $\sigma(G, A', L', H, v) \leq \sigma(v)$. Moreover, since for each $i \in \set{0,1}$, the set $\set{x_0, x_1, x_2}$ contains two distinct neighbors of $y_i$, we have
				$\sigma(G, A', L', H, v_i) \leq \sigma(v_i) - 1$.
				Therefore, $\sigma(G, A', L', H) \leq \sigma - 2$, and thus
				$$
					\frac{3}{2} \sigma(G, A', L', H) \leq \frac{3(\sigma - 2)}{2} \leq |A| - 3 \leq |A'|.
				$$
				By the minimality of $|A|$, the conclusion of Theorem~\ref{theo:upper_bound} holds for $(k, G, A', L', H)$, i.e., the graph $\J(G, A')$ is $(L', H)$-colorable. By the definition of~$L'$, for any $(L', H)$-coloring $I$ of $\J(G, A')$, $I \cup \set{x_0,x_1, x_2}$ is an $(L,H)$-coloring of $\J(G, A)$. This is a~contradiction.
			\end{claimproof}
			
			Due to Corollary~\ref{corl:k_geq_2}, we can choose a pair of disjoint independent sets $U_0$, $U_1 \subset V(G)$ such that $\chi(G - U_0) = \chi(G - U_1) = k-1$. Choose arbitrary elements $a_1 \in A$ and $x_1 \in L(a_1)$. By Claim~\ref{claim:no_neighbor}, for each $u \in U_0 \cup U_1$, there is a unique element $y(u) \in L(u)$ adjacent to $x_1$ in $H$ (the uniqueness of $y(u)$ follows from the definition of a cover). Let
			$$
				I_0 \coloneqq \set{y(u) \,:\, u \in U_0} \;\;\;\text{ and }\;\;\; I_1 \coloneqq \set{y(u) \,:\, u \in U_1}.
			$$
			Since $U_0$ and $U_1$ are independent sets in $G$, $I_0$ and $I_1$ are independent sets in $H$.
			
			\begin{claim}\label{claim:first_step}
				There exists an element $a_0 \in A \setminus \set{a_1}$ such that $L(a_0) \cap N_H(I_0) \not \subseteq N_H(x_1)$.
			\end{claim}
			\begin{claimproof}
				Assume that for all $a \in A \setminus \set{a_1}$, we have $L(a) \cap N_H(I_0) \subseteq N_H(x_1)$. Let $G' \coloneqq G - U_0$, and for each $w \in V(G') \cup A$, let $L'(w) \coloneqq L(w) \setminus N_H(I_0)$. By the definition of $I_0$, $L'(a_1) = L(a_1) \setminus \set{x_1}$, so $$|L'(a_1)| = |L(a_1)| - 1 \geq |A| + (k-1).$$
				On the other hand, by our assumption, for each $a \in A \setminus \set{a_1}$, we have $$|L'(a)| = |L(a) \setminus N_H(I_0)| \geq |L(a) \setminus N_H(x_1)| \geq |L(a)| - 1 \geq |A| + (k-1).$$
				Since for all $v \in V(G)$, $\sigma(G', A, L', H, v) \leq \sigma(v)$, the minimality of $k$ implies the conclusion of Theorem~\ref{theo:upper_bound} for $(k-1, G', A, L', H)$; in other words, the graph $\J(G', A)$ is $(L', H)$-colorable. By the definition of~$L'$, for any $(L', H)$-coloring $I$ of $\J(G', A)$, $I \cup I_0$ is an $(L,H)$-coloring of $\J(G, A)$; this is a~contradiction.
			\end{claimproof}
			
			Using Claim~\ref{claim:first_step}, fix some $a_0 \in A \setminus \set{a_1}$ satisfying $L(a_0) \cap N_H(I_0) \not\subseteq N_H(x_1)$, and choose any
			$$x_0 \in (L(a_0) \cap N_H(I_0)) \setminus N_H(x_1).$$
			Since $x_0 \in N_H(I_0)$, we can also choose $y_0 \in I_0$ so that $x_0 y_0 \in E(H)$.
			
			\begin{claim}
				$x_0 \not \in N_H(I_1)$.
			\end{claim}
			\begin{claimproof}
				If there is $y_1 \in I_1$ such that $x_0 y_1 \in E(H)$, then $x_0 - y_0 - x_1 - y_1 - x_0$ is a walk in $H$ whose existence is ruled out by Claim~\ref{claim:walk}.
			\end{claimproof}
			
			\begin{claim}
				There is an element $a_2 \in A \setminus \set{a_0, a_1}$ such that $L(a_2) \cap N_H(I_1) \not \subseteq N_H(x_0, x_1)$.
			\end{claim}
			\begin{claimproof}
				The proof is almost identical to the proof of Claim~\ref{claim:first_step}. Assume that for all $a \in A \setminus \set{a_0, a_1}$, we have $L(a) \cap N_H(I_1) \subseteq N_H(x_0, x_1)$. Let $G' \coloneqq G - U_1$, $A' \coloneqq A \setminus \set{a_0}$, and for each $w \in V(G') \cup A'$, let $L'(w) \coloneqq L(w) \setminus N_H(\set{x_0} \cup I_1)$. By the definition of $I_1$, $L(a_1) \cap N_H(I_1) = \set{x_1}$, so $$|L'(a_1)| \geq |L(a_1)| - 2 \geq |A| + k-2 = |A'| + (k-1).$$
				On the other hand, by our assumption, for each $a \in A \setminus \set{a_0, a_1}$, we have $$|L'(a)| \geq |L(a) \setminus N_H(x_0, x_1)| \geq |L(a)| - 2 \geq |A| + k-2 = |A'| + (k-1).$$
				Since for all $v \in V(G)$, $\sigma(G', A', L', H, v) \leq \sigma(v)$, the minimality of $k$ implies the conclusion of Theorem~\ref{theo:upper_bound} for $(k-1, G', A', L', H)$; in other words, the graph $\J(G', A')$ is $(L', H)$-colorable. By the definition of~$L'$, for any $(L', H)$-coloring $I$ of $\J(G', A)$, $I \cup \set{x_0} \cup I_1$ is an $(L,H)$-coloring of $\J(G, A)$. This is a~contradiction.
			\end{claimproof}
			
			Now we are ready to finish the proof of Theorem~\ref{theo:upper_bound}. Fix some $a_2 \in A \setminus \set{a_0, a_1}$ satisfying $L(a_2) \cap N_H(I_1) \not\subseteq N_H(x_0, x_1)$, and choose any
			$$
				x_2 \in (L(a_2) \cap N_H(I_1)) \setminus N_H(x_0, x_1).
			$$
			Since $x_2 \in N_H(I_1)$, there is $y_1 \in I_1$ such that $x_2 y_1 \in E(H)$. Then $x_0-y_0-x_1-y_1-x_2$ is a walk in $H$ contradicting the conclusion of Claim~\ref{claim:walk}.
		\end{proof}
		
	Now it is easy to derive Theorem~\ref{theo:main}. Indeed, let $G$ be a graph with $n$ vertices, $m$ edges, and chromatic number $k$, let $A$ be a finite set disjoint from $V(G)$, and let $(L, H)$ be a cover of $\J(G, A)$ such that for all $v \in V(G)$ and $a \in A$, $|L(v)| = |L(a)| = \chi(\J(G,A)) = |A| + k$. Note that
	$$
		\frac{3}{2} \sum_{v \in V(G)} \max \set{\deg_G(v) - |L(v)| + |A| + 1,\,0} = \frac{3}{2} \sum_{v \in V(G)} \max \set{\deg_G(v) - k + 1, \, 0}.
	$$
	If $|A| \geq 3m$, then
	$$
		\frac{3}{2} \sum_{v \in V(G)} \max \set{\deg_G(v) - k + 1, \, 0} \leq \frac{3}{2} \sum_{v \in V(G)} \deg_G(v) = 3m \leq |A|,
	$$
	so Theorem~\ref{theo:upper_bound} implies that $\J(G, A)$ is $(L, H)$-colorable, and hence $Z_{DP}(G) \leq 3m$. Moreover, if $\delta(G) \geq k-1$, then
	$$
		\frac{3}{2} \sum_{v \in V(G)} \max \set{\deg_G(v) - k + 1, \, 0}  = \frac{3}{2} \sum_{v \in V(G)} (\deg_G(v) - k + 1)  = 3m - \frac{3}{2}(k-1)n,
	$$
	so $Z_{DP}(G) \leq 3m - \frac{3}{2}(k-1)n$, as desired. Finally, Corollary~\ref{corl:Z_n} follows from Theorem~\ref{theo:main} and the fact that an $n$-vertex graph can have at most ${n \choose 2} \leq n^2/2$ edges.
	
\section{Proof of Theorem~\ref{theo:lower}}\label{sec:low}
We will prove the following precise version of Theorem~\ref{theo:lower}:

\begin{theo}\label{theo:lower2}
			For all even $n \in \N$,
			$Z_{DP}(n) \geq n^2/4-n$.	
		\end{theo}			
		\begin{proof} Let $n\in \N$ be even and let $k \coloneqq n/2-1$. Note that $n^2/4 - n = k^2 - 1$. Thus, it is enough to exhibit an $n$-vertex bipartite graph $G$ and a cover $(L, H)$ of $\J(G, k^2 - 2)$ such that $|L(u)| = k^2$ for all $u \in V(\J(G, k^2 - 2))$, yet $\J(G, k^2 - 2)$ is not $(L, H)$-colorable.
			 
		 Let $G \cong K_{n/2, n/2}$ be an $n$-vertex complete bipartite graph with parts $X = \set{x, x_0, \ldots, x_{k-1}}$ and $Y = \set{y, y_0, \ldots, y_{k-1}}$, where the indices $0$, \ldots, $k-1$ are viewed as elements of the additive group $\Z_k$ of integers modulo $k$. Let $A$ be a set of size $k^2 - 2$ disjoint from $X \cup Y$. For each $u \in X \cup Y \cup A$, let $L(u) \coloneqq \set{u} \times \Z_k \times \Z_k$. Let $H$ be the graph with vertex set $(X \cup Y \cup A) \times \Z_k \times \Z_k$ in which the following pairs of vertices are adjacent:
		 
		 \begin{itemize}
		 	\item[--] $(u, i, j)$ and $(u, i', j')$ for all $u \in X \cup Y \cup A$ and $i$, $j$, $i'$, $j' \in \Z_k$ such that $(i, j) \neq (i', j')$;
		 	\item[--] $(u, i, j)$ and $(v, i, j)$ for all $u \in \set{x, y} \cup A$, $v \in N_{\J(G, A)}(u)$, and $i$, $j \in \Z_k$;
		 	\item[--] $(x_s, i, j)$ and $(y_t, i+s, j+t)$ for all $s$, $t$, $i$, $j \in \Z_k$.
		 \end{itemize}
			
			It is easy to see that $(L, H)$ is a cover of $\J(G, A)$. We claim that $\J(G, A)$ is not $(L, H)$-colorable. Indeed, suppose that $I$ is an $(L, H)$-coloring of $\J(G, A)$. For each $u \in X \cup Y \cup A$, let $i(u)$ and $j(u)$ be the unique elements of $\Z_k$ such that $(u, i(u), j(u)) \in I$. By the construction of $H$ and since $I$ is an independent set, we have
			$$
				(i(u), j(u)) \neq (i(a), j(a))
			$$
			for all $u \in X \cup Y$ and $a \in A$. Since all the $k^2 - 2$ pairs $(i(a), j(a))$ for $a \in A$ are pairwise distinct, $(i(u), j(u))$ can take at most $2$ distinct values as $u$ is ranging over $X \cup Y$. One of those $2$ values is $(i(y), j(y))$, and if $u \in X$, then
			$$
				(i(u), j(u)) \neq (i(y), j(y)),
			$$
			so the value of $(i(u), j(u))$ must be the same for all $u \in X$; let us denote it by $(i, j)$. Similarly, the value of $(i(u), j(u))$ is the same for all $u \in Y$, and we denote it by $(i', j')$.
			
			It remains to notice that the vertices $(x_{i' - i}, i, j)$  and $(y_{j' - j}, i', j')$ are adjacent in $H$, so $I$ is not an independent set.
		\end{proof}
		
		Now we can prove Corollary~\ref{corl:NRW_DP}:
		
		\begin{proof}[Proof of Corollary~\ref{corl:NRW_DP}]
			First, suppose that $G$ is an $n$-vertex graph with $\chi(G) = r$ that maximizes the difference $\chi_{DP}(G) - \chi(G)$. Adding edges to $G$ if necessary, we may arrange $G$ to be a complete $r$-partite graph. Assuming $2r > n$, at least $2r - n$ of the parts must be of size $1$, i.e., $G$ is of the form $\J(G', 2r-n)$ for some $2(n-r)$-vertex graph $G'$. By Corollary~\ref{corl:Z_n}, we have $\chi_{DP}(G) = \chi(G)$ as long as $2r - n \geq 6(n-r)^2$, which holds for all $r \geq n - (1/\sqrt{6}-o(1))\sqrt{n}$. This establishes the upper bound $r(n) \leq n - \Omega(\sqrt{n})$.
			
			On the other hand, due to Theorem~\ref{theo:lower}, for each $n$, we can find a graph $G$ with $s$ vertices, where $s \leq (2+o(1))\sqrt{n}$, such that $\chi_{DP}(\J(G, n - s)) > \chi(\J(G, n - s))$. Since $\J(G, n - s)$ is an $n$-vertex graph, we get
			$$
				r(n) > \chi(\J(G, n - s)) = \chi(G) + n - s \geq n - (2+o(1))\sqrt{n} = n - O(\sqrt{n}). \qedhere
			$$
		\end{proof}
		
		\paragraph{Acknowledgements.} The authors are grateful to the anonymous referees for their valuable comments and suggestions.

\end{document}